\documentclass{elsarticle}
\usepackage{hyperref}
\usepackage[utf8]{inputenc}
\usepackage{amsfonts}
\usepackage{psfrag}
\usepackage{subfigure}

\journal{Journal of Combinatorial Theory - Series B}

\newtheorem{thm}{Theorem}[section]
\newtheorem{lem}[thm]{Lemma}
\newtheorem{cor}[thm]{Corollary}
\newtheorem{exa}[thm]{Example}
\newtheorem{pro}[thm]{Proposition}
\newtheorem{conj}[thm]{Conjecture}

\newtheorem{cas}{Case}
\newtheorem{sub}{Case}[cas]

\newlength{\blank}
\newtheorem{sta}{\hspace{-\blank}}[thm] % statement inside a theorem
\newcommand{\Contra}[3]{#1 / (#2 \rightarrow #3)}

\newcommand{\rmenum}
{
\renewcommand{\theenumi}{{\rm(\roman{enumi})}}
\renewcommand{\labelenumi}{\theenumi}
}

\newenvironment{proof}{\noindent\underline{Proof}:}{\hfill\qed\par\medskip}

\begin{document}
  \title{Laminar Tight Cuts in Matching Covered Graphs\tnoteref{mytitlenote}}
  \tnotetext[mytitlenote]{Research Supported by NSF grant DMS-1855716, National
    Natural Science Foundation    in China (Grants No. 11671186 and 11471273),
    NSF of Shandong Province and Brazilian \textsc{cnp}q}

  %% \author{Elsevier\fnref{myfootnote}}
  %% \address{Radarweg 29, Amsterdam}
  %% \fntext[myfootnote]{Since 1880.}

  \author[chenaddress]{Guantao Chen}
  \ead{gchen@gsu.edu}
  \address[chenaddress]{Department of Mathematics and Statistics, Georgia State
    University, Atlanta, GA 30303}

  \author[fengaddress]{Xing Feng\corref{corresponding}}
  \ead{fengxing\_fm@163.com}
  \address[fengaddress]{Faculty of Science, Jiangxi University of Science and Technology, Ganzhou 341000, China}

  \author[luaddress]{Fuliang Lu}
  \ead{flianglu@163.com}
  \address[luaddress]{School of Mathematics and Statistics,
  Minnan Normal University, Zhangzhou 363000,  China}

  \author[lucchesiaddress]{Cláudio L.~Lucchesi}
  \ead{lucchesi@ic.unicamp.br}
  \address[lucchesiaddress]{Institute of Computing, \textsc{unicamp}, Campinas, Brasil}

  \author[zhangaddress]{Lianzhu Zhang}
  \ead{zhanglz@xmu.edu.cn}
  \address[zhangaddress]{School of Mathematical Science, Xiamen University, Xiamen 361005, China}

  \cortext[corresponding]{Corresponding author}

  \begin{frontmatter}

    \begin{abstract}
      An edge cut $C$ of a graph  $G$ is {\it tight} if $|C \cap M|=1$
      for  every  perfect  matching   $M$  of  $G$.~Barrier  cuts  and
      2-separation  cuts  are called  {\it  ELP-cuts},  which are  two
      important   types   of   tight    cuts   in   matching   covered
      graphs.~Edmonds,  Lov\'asz  and  Pulleyblank proved  that  if  a
      matching covered graph has a  nontrivial tight cut, then it also
      has a  nontrivial ELP-cut.~Carvalho, Lucchesi, and  Murty made a
      stronger conjecture:  given any  nontrivial tight  cut $C$  in a
      matching covered  graph $G$,  there exists a  nontrivial ELP-cut
      $D$ in $G$  which does not cross $C$.~We  confirm the conjecture
      in this paper.
    \end{abstract}


  \end{frontmatter}

  \section{Introduction}

  All  graphs considered  in this  paper  are finite  and may  contain
  multiple edges, but no loops.  We will generally follow the notation
  and terminology used by Bondy and Murty in~\cite{bm08}.

  Let $G$ be  a graph with vertex  set $V$ and edge set  $E$.  For any
  $X\subseteq V$, let  $N_G(X)$ be the set of vertices  of $V-X$ which
  are adjacent to vertices of $X$  and let $\overline{X} = V-X$ be the
  complement of $X$ in $V$.  The set of {\it coboundary edges} of $X$,
  denoted by $\partial(X)$, is the set  of edges with exact one end in
  $X$ and one end in $\overline{X}$.  A {\it cut} is a coboundary edge
  set.   We   call  $X$  and   $\overline{X}$  the  {\it   shores}  of
  $\partial(X)$.  A cut  $\partial(X)$ is {\it trivial}  if $|X|=1$ or
  $|\overline{X}| =1$.

  Let   $C:=\partial(X)$   be   a   cut  of   $G$.   We   denote   by
  $\Contra{G}{X}{x}$ the graph obtained from $G$ by contracting $X$ to
  a single vertex $x$ (and  removing any resulting loops).  The graphs
  $\Contra{G}{X}{x}$ and  $\Contra{G}{\overline{X}}{\overline{x}}$ are
  the two \emph{$C$-contractions of $G$}.

  Let $C:=\partial(X)$  and $D:=\partial(Y)$ be  two cuts of  $G$.  We
  say  that $C$  and  $D$ {\it  cross}  if all  the  four sets  $X\cap
  Y,\overline{X}\cap         Y,         X\cap\overline{Y}$         and
  $\overline{X}\cap\overline{Y}$ are  nonempty, and are  {\it laminar}
  otherwise.  So, two cuts $C$ and $D$  are laminar if and only if one
  of the two shores of $C$ is a  subset of one of the shores of $D$. A
  collection of cuts is \emph{laminar} if no two of its cuts cross.

  A graph is called {\it matching  covered} if it is connected, has at
  least one  edge and each of  its edges is contained  in some perfect
  matching. Suppose that our graph $G$ is matching covered.  A cut $C$
  of $G$  is {\it tight} if  $|C\cap M|=1$ for every  perfect matching
  $M$ of  $G$. Clearly every  trivial cut is a  tight cut.  We  call a
  matching covered graph which is free of nontrivial tight cuts a {\it
    brace} if it is bipartite, and a {\it brick} otherwise.

  Let  $S$ be  a set  of vertices  of $G$.   A component  of $G-S$  is
  \emph{odd}  if  it  consists  of  an  odd  number  of  vertices,  is
  \emph{even} if  it consists of  an even  number of vertices,  and is
  \emph{trivial}  if it  consists of  only one  vertex.  We  denote by
  $o(G-S)$  the number  of odd  components of  $G-S$.

  We  shall make  use  of  the following  known  facts about  matching
  covered graphs and tight cuts.

  \begin{thm}[Tutte~\cite{Tutte47}]
    \label{thm:tutte}
    A graph $G$ has a perfect  matching if and only if $o(G-S)\le |S|$
    for every subset $S$ of $V(G)$.
 \end{thm}

  \begin{cor}
    \label{cor:barrier}
    Let $G$  be a matching  covered graph and let  $S$ be a  subset of
    $V(G)$.  Then, $o(G-S)  \le |S|$,  with  equality only  if $S$  is
    independent and $G-S$ has no even components.
  \end{cor}
  \begin{proof}
    A matching covered graph has perfect matchings, hence the asserted
    inequality holds, by Theorem~\ref{thm:tutte}.  Suppose that $S$ is
    not independent, let $e:=vw$ be an edge of $G$ having both ends in
    $S$. As $G$  is matching covered, it has a  perfect matching, $M$,
    that contains the  edge $e$. Let $T:=S-\{v,w\}$. Then,  $M-e$ is a
    perfect  matching of  $H:=G-\{v,w\}$. By  Theorem~\ref{thm:tutte},
    $o(G-S) = o(H-T)  \le |T| = |S|-2$. The inequality  is thus strict
    in  this case.  Now  suppose  that $G-S$  has  an even  component,
    $K$. As $G$ is connected, $K$ has a vertex, $x$, which is adjacent
    to a  vertex, $y$ of  $S$. Let  $U:=S+x$.  Every odd  component of
    $G-S$ is  an odd  component of  $G-U$. In  addition, $K-x$  has at
    least  one  odd component  which  is  not  a component  of  $G-S$.
    Moreover, the set $U$ is not independent, hence $o(G-S) \le o(G-U)
    -1 \le |U|-3 = |S|-2$.
  \end{proof}

  \begin{pro}[\cite{ELP82}]
    \label{pro:I-U}
    Let  $G$ be  a matching  covered graph  and let  $\partial(X)$ and
    $\partial(Y)$  be  two  tight  cuts  such that  $|X  \cap  Y|$  is
    odd. Then $\partial(X  \cap Y)$ and $\partial(X \cup  Y)$ are also
    tight in $G$. Furthermore, no edge connects $X\cap\overline{Y}$ to
    $\overline{X}\cap Y$. \qed
  \end{pro}

  \begin{pro}[\cite{LP86}]
    \label{pro:mc=>2-connected}
    Every  matching  covered  graph  on   four  or  more  vertices  is
    2-connected. \qed
  \end{pro}

  \begin{pro}[\cite{LOVASZ1987187}]
    \label{pro:tight-contraction}
    Let $G$ be a matching covered graph, and let $C$ be a tight cut of
    $G$. Then  both $C$-contractions are matching  covered.  Moreover,
    if $G'$ is  a $C$-contraction of $G$  then a tight cut  of $G'$ is
    also a tight cut  of $G$.  Conversely, if a tight cut  of $G$ is a
    cut of $G'$ then it is also tight in $G'$. \qed
  \end{pro}

  \begin{cor}\label{cor:shores-connected}
    Let $G$ be a matching covered  graph, and let $C=\partial(X)$ be a
    tight cut of $G$.  Then, both shores $X$ and $\overline{X}$ of $C$
    induce connected graphs.
  \end{cor}
  \begin{proof}
    The  $C$-contraction  $G':=\Contra{G}{X}{x}$  of $G$  is  matching
    covered,       hence       it       is       2-connected,       by
    Proposition~\ref{pro:mc=>2-connected}.  Thus, $G'-x$ is connected.
    In  other words,  $\overline{X}$ induces  a connected  subgraph of
    $G$. Likewise, $X$ also induces a connected subgraph of $G$.
  \end{proof}

  If $C$ is a tight cut of $G$, then both  $C$-contractions of $G$ are
  matching covered.   A nontrivial tight cut  may help us to  reduce a
  matching covered graph  to smaller matching covered  graphs.  We may
  apply to $G$ a procedure, called a \emph{tight cut decomposition} of
  $G$, which produces a list of bricks and braces.  If $G$ itself is a
  brick or a brace then the list consists of just $G$.  Otherwise, $G$
  has a nontrivial tight cut, $C$.  Then, both $C$-contractions of $G$
  are  matching covered.   One  may recursively  apply  the tight  cut
  decomposition  procedure to  each $C$-contraction  of $G$,  and then
  combine the resulting lists to  produce a tight cut decomposition of
  $G$ itself.

  \begin{thm}[Lovász~\cite{LOVASZ1987187}]
    Any two applications  of the tight cut  decomposition procedure to
    $G$ produce  the same list  of bricks  and braces, up  to multiple
    edges. \qed
  \end{thm}
  In particular, any  two applications of the  tight cut decomposition
  procedure yield the same number of bricks, which is called the {\it
    brick number} of $G$ and denoted by $b(G)$.

  The graph $G$ is \emph{bicritical}  if $G-S$ has a perfect matching,
  for each pair $S$ of vertices of  $G$. If $G$ is bicritical then, by
  Tutte's Theorem, every barrier of $G$ is trivial.

  \subsection{ELP-cuts}
  There  are two  types of  tight cuts  that play  a critical  role in
  studying matching theory.  A {\it barrier}  of $G$ is a nonempty set
  of vertices of $G$ such that  $o(G-B)=|B|$.  Moreover, $B$ is a {\it
    nontrivial  barrier}  if  $|B|\ge2$.   By  Tutte's  theorem,  each
  barrier $B$  of $G$  is an  independent set,  and all  components of
  $G-B$ are odd components. A cut $C$ is called a {\it barrier-cut} if
  there exists  a barrier $B$ and  a component $H$ of  $G-B$ such that
  $C=\partial(V(H))$. Clearly, a barrier-cut is a tight cut.

  A \emph{2-separation} of  $G$ is a pair $S$ of  vertices of $G$ such
  that $G-S$ is  not connected and each of the  components of $G-S$ is
  even. Let $\{u,v\}$ be a 2-separation of $G$, let $\{G_1,G_2\}$ be a
  partition    of    $G-\{u,v\}$.     Each    of    the    two    cuts
  $C':=\partial(V(G_1)+u)$  and  $C'':=\partial(V(G_1)+v)$ is  a  {\it
    2-separation  cut}  of   $G$  and  the  pair   $\{C',C''\}$  is  a
  \emph{2-separation cut pair} of $G$.

  Barrier-cuts  and 2-separation  cuts are  particular types  of tight
  cuts and are called {\it ELP cuts}, named after Edmonds, Lovász, and
  Pulleyblank, who proved  the following fundamental result.

  \begin{thm}[The ELP Theorem~\cite{ELP82}]
    \label{thm:elp}
    Every matching covered graph that has a nontrivial tight cut has
    either a nontrivial barrier or a 2-separation. \qed
  \end{thm}
  Their  proof  is  based on  linear  programming  techniques. Szigeti
  \cite{Szigeti02} gave  a purely graph theoretical  proof.  Carvalho,
  Lucchesi and Murty \cite{CLM18} provided an alternative proof.

  \subsection{Ultimately ELP cuts}
  We now define  a special type of tight  cut decomposition procedure.
  We  may  apply   to  $G$  a  procedure,  called   an  \emph{ELP  cut
    decomposition} of $G$, which produces a list of bricks and braces.
  If~$G$ itself is a  brick or a brace then the  list consists of just
  $G$.  Otherwise, by  the ELP Theorem, $G$ has a  nontrivial ELP cut,
  $C$.  Then, both $C$-contractions of  $G$ are matching covered.  One
  may recursively  apply the ELP  cut decomposition procedure  to each
  $C$-contraction  of $G$,  and then  combine the  resulting lists  to
  produce an  ELP cut decomposition of  $G$ itself.  Each cut  used in
  the ELP cut decomposition procedure  is said to be \emph{ultimately}
  an ELP cut.  It is easy to see that the following conjecture implies
  that every nontrivial tight cut of $G$ is ultimately an ELP cut.

  \begin{conj}{\em [Carvalho, Lucchesi and Murty \cite{CLM18}]}
    \label{conj}
    Let $C$  be a  nontrivial tight  cut of  a matching  covered graph
    $G$. Then, $G$ has a nontrivial ELP cut which is laminar with $C$.
  \end{conj}

  In 2002,  the three authors  proved a particular and  very important
  case of the Conjecture,  in which $G$ is a brick and  $e$ is an edge
  of    $G$    such   that    $G-e$    is    matching   covered    and
  $b(G-e)=2$~\cite{CLM022}.   In~\cite{CLM18}  they  also  proved  the
  validity of the Conjecture for  bicritical graphs.

  In this paper we present a proof of a result, our Main Theorem, that
  implies Conjecture  \ref{conj}.  To state  the Main Theorem  we need
  one more definition and a simple result.

  Let $G$ be a matching covered graph, let $C$ be a tight   cut of $G$
  and let $S$  be a set of  vertices of $G$ which  is either a barrier
  or a 2-separation. The set $S$ is  \emph{$C$-sheltered} if $S$ is a
  subset of  a shore  of  $C$,  and  is \emph{$C$-avoiding}  if,  each
  ELP  cut associated     with $S$ is  laminar     with $C$. If $S$ is
  $C$-sheltered then, as each shore of $C$ induces a connected subgraph
  $H$  of  $G$ (Corollary~\ref{cor:shores-connected}), one   of the
  components of $G-S$ is a supergraph of $H$. It follows that if $S$ is
  $C$-sheltered then $S$ is $C$-avoiding. We record this    result for
  later reference.

  \begin{pro}
    \label{pro:sheltered}
    Let  $S$ be  either  a 2-separation  or a  barrier  of a  matching
    covered graph  $G$ and let $C$  be a tight  cut of $G$. If  $S$ is
    $C$-sheltered then some  cut associated with $S$ has  a shore that
    is a superset of a shore of  $C$, say, $X$, and all the other cuts
    associated  with   $S$  have   a  shore  that   is  a   subset  of
    $\overline{X}$. Consequently,  if $S$ is $C$-sheltered  then it is
    $C$-avoiding. \qed
  \end{pro}

  We now  state our result, which,  in view of the  Proposition above,
  implies Conjecture~\ref{conj}.

  \begin{thm}[Main Theorem]
    \label{thm:laminar-elp}
    Let $C$  be a  nontrivial tight  cut of  a matching  covered graph
    $G$. Then either  $G$ has a $C$-sheltered nontrivial  barrier or a
    2-separation    cut   which    is    laminar    with   $C$    (see
    Example~\ref{exa:example-main}).
  \end{thm}

  \begin{exa}
    \label{exa:example-main}
    Consider the graph  depicted in Figure~\ref{fig:example-main}. The
    tight cut $C$ is laminar with the cut $D$, which is a 2-separation
    cut associated  with the pair  $\{u_1,u_2\}$. The cut $C$  is also
    laminar with  the 2-separation cut  $F$, which is  associated with
    the   pair  $\{1,b_3\}$.    The  barriers   $\{b_1,b_2,b_3\}$  and
    $\{2,3\}$ are  $C$-sheltered, whereas  the barrier  $\{1,2,3\}$ is
    not $C$-avoiding.
  \end{exa}

  \begin{figure}[!ht]
    \centering
    \psfrag{C}{$C$}
    \psfrag{D}{$D$}
    \psfrag{F}{$F$}
    \psfrag{u1}{$u_1$}
    \psfrag{u2}{$u_2$}
    \psfrag{b1}{$b_1$}
    \psfrag{b2}{$b_2$}
    \psfrag{b3}{$b_3$}
    \psfrag{1}{$1$}
    \psfrag{2}{$2$}
    \psfrag{3}{$3$}
    \psfrag{4}{}
    \psfrag{5}{}
    \psfrag{6}{}
    \psfrag{7}{}
    \includegraphics [width = .65\textwidth] {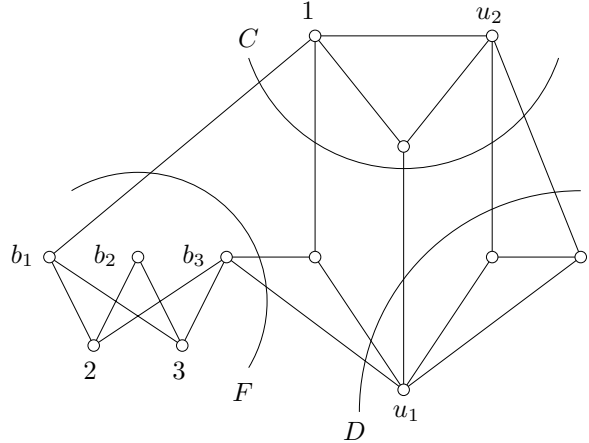}
    \caption{The  graph of Example~\ref{exa:example-main}}
    \label{fig:example-main}
  \end{figure}

  The proof of the Main Theorem will be given in Section \ref{sec:pf},
  after we present some necessary results in Section \ref{sec:pre}

  \section{Ingredients}
  \label{sec:pre}

  \noindent
  In this section we prove three Lemmas, which play a crucial role in
  the proof of the Main Theorem.

  \begin{lem}
    \label{lem:inheritance}
    ~Let $C:=\partial(X)$ be a 2-separation cut of a matching covered
    graph~$G$,  associated with  a  2-separation $\{u_1,u_2\}$,  where
    $u_1 \in X$ and $u_2 \in \overline{X}$.  Let $S_H$ denote either a
    2-separation    or    a    barrier    of    the    $C$-contraction
    $H:=\Contra{G}{\overline{X}}{\overline{x}}$ of $G$ and let
    \begin{displaymath}
      S:= \left \{
      \begin{array}{rl}
 	S_H,    &     \mbox{if    $\overline{x}     \notin    S_H$},\\
 	(S_H-\overline{x})+u_2, & \mbox{if $\overline{x} \in S_H$}.
      \end{array}
      \right .
    \end{displaymath}
    Then      the       following      properties       hold      (See
    Figure~\ref{fig:inheritance}):
    \begin{figure}[!ht]
      \centering
      \psfrag{C}{$C$}
      \psfrag{X}{$X-u_1$}
      \psfrag{Xb}{$\overline{X}-u_2$}
      \psfrag{u1}{$u_1$}
      \psfrag{u2}{$u_2$}
      \psfrag{xb}{$\overline{x}$}
      \psfrag{G}{{\small(a) $G$}}
      \psfrag{H}{{\small(b) $H$}}
      \includegraphics [width = \textwidth] {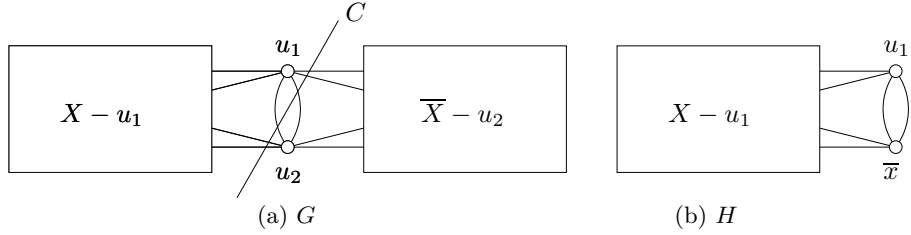}
      \caption{Illustration for Lemma~\ref{lem:inheritance}}
      \label{fig:inheritance}
    \end{figure}
    \begin{enumerate}\rmenum
    \item
      \label{itm:none-in-u1-xbar}
      every component of $H-S_H$ that does not contain vertices in
      $\{u_1,\overline{x}\}$ is a component of $G-S$,
    \item
      \label{itm:<2-components}
      at  most   one  component   of  $H-S_H$  contains   vertices  in
      $\{u_1,\overline{x}\}$, and
    \item
      \label{itm:inheritance}
      if $S_H$ is a barrier of $H$ then $S$ is a barrier of $G$ and if
      $S_H$ is  a 2-separation of  $H$ then  $S$ is a  2-separation of
      $G$.
    \end{enumerate}
  \end{lem}
  \begin{proof}
    Let $\mathcal{H}$  denote the collection of  components of $H-S_H$
    and  let $\mathcal{H}_0$  denote the  collection of  components of
    $H-S_H$ that contain a least one vertex in $\{u_1,\overline{x}\}$.

    \bigskip
    \noindent
    \underline{\ref{itm:none-in-u1-xbar}}: Let  $H[Y]$ be  a component
    in    $\mathcal{H}   -    \mathcal{H}_0$.    By    definition   of
    $\mathcal{H}_0$,       $Y       \subseteq      X-u_1$,       hence
    $\partial_G(Y)=\partial_H(Y)$.  We shall now prove that every edge
    of $\partial(Y)$ is incident with a vertex of $S$.

    Let $e$ be  an edge of $\partial(Y)$. Then $e$  joins a vertex $u$
    in $Y$ to a  vertex $v$ in $S_H$. If $v  \ne \overline{x}$ then $v
    \in S$.  We may thus assume that $v = \overline{x}$, which implies
    that $e \in C$. Every edge of $C$ is incident in $G$ with a vertex
    in $\{u_1,u_2\}$.  Thus,  either~$u_1$ is an end of $e$  in $X$ or
    $u_2$ is  an end  of $e$  in $\overline{X}$. As  $Y \subseteq  X -
    u_1$, the end of  $e$ in $Y$ is not~$u_1$. It  follows that $e$ is
    not incident with $u_1$, hence  $e$ is incident with~$u_2$ in~$G$.
    In  both alternatives  we conclude  that  $e$ is  incident with  a
    vertex  of  $S$.  This  conclusion  holds  for  each edge  $e  \in
    \partial(Y)$.  We deduce that $H[Y]$ is a component of $G-S$.

    \bigskip
    \noindent
    \underline{\ref{itm:<2-components}}: The graph  $G$ is 2-connected
    (Proposition~\ref{pro:mc=>2-connected}), hence  the vertices $u_1$
    and $\overline{x}$ are adjacent in the graph $H$.

    \begin{sta}
      \label{sta:<2-components}
      $|\mathcal{H}_0| \le 1$, with equality  if $S_H$ is a barrier of
      $H$.
    \end{sta}
    \begin{proof}
      If neither $u_1$  nor $\overline{x}$ is in $S_H$  then, as $u_1$
      and    $\overline{x}$   are    adjacent,    it   follows    that
      $|\mathcal{H}_0|=1$.  Assume thus that at least one of $u_1$ and
      $\overline{x}$  is   in~$S_H$.   In  that  case,   the  asserted
      inequality  holds.  Moreover,  suppose that  $S_H$ is  a barrier
      then, by  Corolary~\ref{cor:barrier}, $S_H$ is  independent.  It
      follows    that   precisely    one    of    the   vertices    in
      $\{u_1,\overline{x}\}$ is in $S_H$, hence equality holds.
    \end{proof}

    \bigskip
    \noindent
    \underline{\ref{itm:inheritance}}: By~(\ref{itm:none-in-u1-xbar}),
    all the  components in $\mathcal{H}-\mathcal{H}_0$  are components
    of $G-S$.  By~(\ref{sta:<2-components}), $|\mathcal{H}_0|  \le 1$,
    with equality if $S_H$ is a barrier of $H$.

    Consider first the  case in which $S_H$ is a  barrier of $H$. Each
    of the $|S|-1$ (odd)  components in $\mathcal{H}-\mathcal{H}_0$ is
    a component  of $G-S$.  By  parity, $G-S$  has at least  $|S|$ odd
    components.  By Corollary~\ref{cor:barrier},  $G-S$ has  precisely
    $|S|$ components, all  of which are odd. Indeed, $S$  is a barrier
    of $G$.

    Finally,  suppose that  $S_H$ is  a 2-separation.   The collection
    $\mathcal{H}-\mathcal{H}_0$  contains a  graph, $K$,  which is  an
    even  component of  $H-S_H$. Thus,  $K$  is an  even component  of
    $G-S$. Moreover, $K$ is a  subgraph of $G[X-u_1]$, the subgraph of
    $G$ induced  by $X-u_1$, hence  $G-S$ has two or  more components.
    As $K$ is even and $G$  is matching covered, it follows that $G-S$
    has only  even components  (Corollary~\ref{cor:barrier}).  Indeed,
    $S$ is a 2-separation of $G$.
  \end{proof}

  Let $G$ be a matching covered graph, let $C:=\partial(X)$ be a tight
  cut of  $G$, let~$B$ denote a  barrier of $G$ and  let $\mathcal{H}$
  denote the set  of components of $G-B$.  For each  shore $Z$ of $C$,
  let  $\mathcal{H}_Z$  be   the  set  of  those   components  $H  \in
  \mathcal{H}$   such  that   $|V(H)  \cap   Z|$  is   odd.   Clearly,
  $|\mathcal{H}_X|  + |\mathcal{H}_{\overline{X}}|  = |\mathcal{H}|  =
  |B|$. See Figure~\ref{fig:c-sheltered}.

  \begin{figure}[!ht]
    \centering
    \psfrag{X}{$X$}
    \psfrag{Y}{$Y$}
    \psfrag{C}{$C$}
    \psfrag{o}{odd}
    \psfrag{e}{even}
    \psfrag{b1}{$b_1$}
    \psfrag{b2}{$b_2$}
    \psfrag{b3}{$b_3$}
    \includegraphics [width=.8\textwidth] {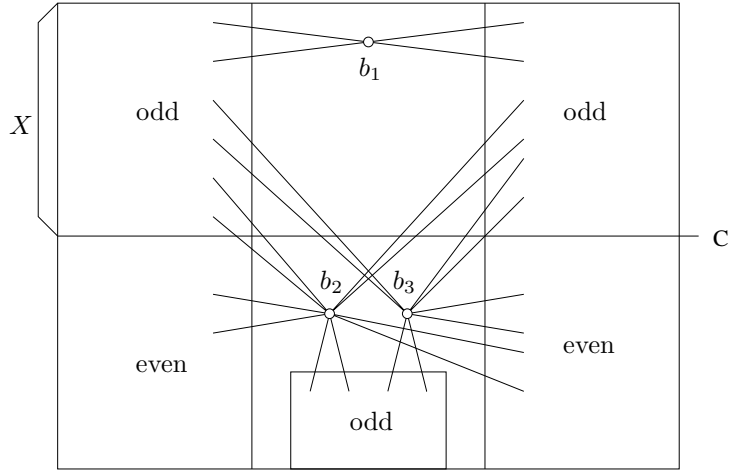}
    \caption{A barrier $B:=\{b_1,b_2,b_3\}$, where $|\mathcal{H}_X|=2$
      and $|\mathcal{H}_{\overline{X}}|=1$}
    \label{fig:c-sheltered}
  \end{figure}

  \begin{lem}
    \label{lem:barrier-sheltered}
    Let $G$  be a  matching covered graph,  let $C:=\partial(X)$  be a
    tight cut of $G$, let $B$ denote a barrier of $G$ and let $K$ be a
    component of $\mathcal{H}_X$ that contains  a vertex adjacent to a
    vertex in $B \cap \overline{X}$.  The following properties hold:
    \begin{enumerate}\rmenum
    \item
      \label{itm:barrier-accounting}
      $|B \cap X| = |\mathcal{H}_X| -  1$ and $|B \cap \overline{X}| =
      |\mathcal{H}_{\overline{X}}| + 1$,
    \item
      \label{itm:barrier-laminar}
      each component in $\mathcal{H}_{\overline{X}}$  is a subgraph of
      $G[\overline{X}]$ and has  no vertex adjacent to a  vertex of $B
      \cap X$,
    \item
      \label{itm:barrier-sheltered}
      $B  \cap \overline{X}$  is a  ($C$-sheltered, possibly  trivial)
      barrier of $G$, and
    \item
      \label{itm:barrier-sheltered-nontrivial}
      if $C$  is nontrivial  and $B$  is $C$-avoiding  and nontrivial,
      then  $B  \cap  \overline{X}$ is  a  ($C$-sheltered)  nontrivial
      barrier of $G$.
    \end{enumerate}
  \end{lem}
  \begin{proof}
    The following  simple statement is  important in the proof  of the
    Lemma.

    \begin{sta}
      \label{sta:edge-in-the-cut}
      Let $Z$ be a shore of $C$, let $H \in \mathcal{H}_Z$, let $e$ be
      an edge  in $\partial(V(H))$ incident  with a vertex of  $B \cap
      \overline{Z}$  and let  $M$ be  a perfect  matching of  $G$ that
      contains edge  $e$.  Then, the edge  of $M \cap C$  has at least
      one end in $H$.
    \end{sta}
    \begin{proof}  Let  $e:=vw$,  $v  \in  V(H)$,  $w  \in  B  \cap
      \overline{Z}$. If $v  \in Z$ then clearly $e$ is  the edge of $M
      \cap C$.   Assume thus that  $v \in \overline{Z}$.   As $W:=V(H)
      \cap \overline{Z}$ is even, it follows that $M \cap \partial(W)$
      has another edge, $f$.  The edge  $f$ cannot have an end in $B$,
      because~$H$ is matched by $M$ to  the end $w$ of $e$.  Thus, $f$
      has an end in $V(H) \cap Z$,  hence $f$ has both ends in $H$ and
      is the edge of $M \cap C$.
    \end{proof}

    \bigskip
    \noindent
    \underline{\ref{itm:barrier-accounting}}: By hypothesis, $K$ has a
    vertex,  $v$, which  is  adjacent to  a vertex,  $w$,  of $B  \cap
    \overline{X}$. Let $M$ be a  perfect matching of $G$ that contains
    edge $vw$. By~(\ref{sta:edge-in-the-cut}), the edge of $M \cap C$,
    say, $f$,  has at least  one end in  $K$. Clearly, either  $f$ has
    both  ends in  $K$ or  it is  incident with  a vertex  of $B$.  It
    follows  that  except  for  $K$,   all  the  other  components  of
    $\mathcal{H}_X$ are matched  by $M$ to vertices of $B  \cap X$ and
    every component of $\mathcal{H}_{\overline{X}}$  is matched by $M$
    to a vertex of $B\cap \overline{X}$.  The asserted equality holds.

    \bigskip
    \noindent
    \ref{itm:barrier-laminar}:   Assume,   to   the   contrary,   that
    $\mathcal{H}_{\overline{X}}$  has  a  component  that  contains  a
    vertex     adjacent     to     a     vertex     in     $B     \cap
    X$. From~\ref{itm:barrier-accounting},  with the roles of  $X$ and
    $\overline{X}$ interchanged, we deduce that
    \begin{displaymath}
      |B \cap  \overline{X}| = |\mathcal{H}_{\overline{X}}| -  1 \quad
      \mbox{and} \quad |B \cap {X}| = |\mathcal{H}_{{X}}| +1.
    \end{displaymath}
    This           is           a           contradiction           to
    property~\ref{itm:barrier-accounting}. Let  $H$ be a  component in
    $\mathcal{H}_{\overline{X}}$ and assume, to the contrary, that $H$
    has vertices  in $X$.   As $H$  has an odd  number of  vertices in
    $\overline{X}$, it  has an even  number of vertices in  $X$, hence
    $V(H) \cap X$  is a proper subset of $X$.   The subgraph $G[X]$ of
    $G$  is  connected  (Corollary~\ref{cor:shores-connected}),  hence
    some vertex of  $V(H) \cap X$ is  adjacent to a vertex  of $B \cap
    X$, a contradiction.

    \bigskip
    \noindent
    \ref{itm:barrier-sheltered}:   Let   $H$   be   a   component   in
    $\mathcal{H}_{\overline{X}}$.   From~\ref{itm:barrier-laminar}, we
    deduce  that $H$  is a  subgraph of  $G[\overline{X}]$.  Moreover,
    every edge  of $\partial(V(H))$  is incident with  a vertex  of $B
    \cap \overline{X}$.  We conclude that  $H$ is a component of $G-(B
    \cap  \overline{X})$.   This  conclusion  holds for  each  $H  \in
    \mathcal{H}_{\overline{X}}$.    From~\ref{itm:barrier-accounting},
    we  infer that  the  $|B  \cap \overline{X}|  -  1$ components  in
    $\mathcal{H}_{\overline{X}}$ are (odd) components  of $G - (B \cap
    \overline{X})$.   By  parity and  Corollary~\ref{cor:barrier},  $B
    \cap \overline{X}$ is a barrier of $G$.

    \bigskip
    \noindent
    \ref{itm:barrier-sheltered-nontrivial}:
    From~\ref{itm:barrier-sheltered},   we   infer    that   $B   \cap
    \overline{X}$ is  a ($C$-sheltered)  barrier of $G$.  Suppose that
    $C$ is nontrivial and $B$ is $C$-avoiding and nontrivial.  Assume,
    to the  contrary, that $B  \cap \overline{X}$ is trivial,  let $v$
    denote    the   only    vertex   of    $B   \cap    \overline{X}$.
    By~\ref{itm:barrier-accounting}, $\mathcal{H}=\mathcal{H}_X$.

    By hypothesis, $C$  is nontrivial, thus $\overline{X} -  v$ is not
    empty.   Consequently,  some  component of  $G-B$,  $H$,  contains
    vertices     of    $\overline{X}     -    v$.      Moreover,    as
    $\mathcal{H}=\mathcal{H}_X$, the  component $H$  contains vertices
    in $X$. In sum, $V(H)$ contains vertices in both shores of $C$.

    By hypothesis, $B$ is nontrivial. As $v$ is the only vertex of $B$
    in $\overline{X}$, it follows that  $B$ has vertices in $X$. Thus,
    $V(H)$  is  not   a  superset  of  $X$.   As  $v$,   a  vertex  of
    $\overline{X}$,  is  in $B$,  it  follows  that  $V(H)$ is  not  a
    superset of $\overline{X}$. In sum, $V(H)$ is neither a subset nor
    a superset  of any shore  of $C$.  Consequently, the cuts  $C$ and
    $\partial(V(H))$ cross, in a  contradiction to the hypothesis that
    $B$ is  $C$-avoiding.  We conclude  that $B \cap  \overline{X}$ is
    nontrivial.
  \end{proof}

  \begin{lem}
    \label{lem:t}
    Let  $C:=\partial(X)$ be  a  nontrivial tight  cut  of a  matching
    covered   graph    $G$   and    let   $t$    be   a    vertex   of
    $\overline{X}$.  Suppose that  the assertion  of the  Main Theorem
    holds for every graph having  $|V(G)|$ or fewer vertices. Then one
    of the following properties holds:
    \begin{enumerate}\rmenum
    \item
      \label{itm:s-no-t}
      either the  graph $G$ has  a 2-separation that does  not contain
      the vertex $t$, or
    \item
      \label{itm:s-t}
      $G$  has a  2-separation,  $S$, that  contains  the vertex  $t$,
      associated with  a cut $D:=\partial(Y)$, such  that $Y \subseteq
      \overline{X}$, or
    \item
      \label{itm:c-sheltered}
      the graph $G$ has a $C$-sheltered nontrivial barrier.
    \end{enumerate}
  \end{lem}
  \begin{proof}
    By induction on $|V(G)|$. By hypothesis, the assertion of the Main
    Theorem  holds for  $G$.  If  $G$ has  a $C$-sheltered  nontrivial
    barrier then alternative~\ref{itm:c-sheltered} of the statement of
    the Lemma  holds. We may thus  assume that $G$ has  a 2-separation
    $T$, and  an associated  cut $F:=\partial(Z)$ such  that $Z$  is a
    subset  of   a  shore  of   $C$.   If   $t  \notin  T$   then  the
    alternative~\ref{itm:s-no-t} of  the assertion holds. We  may thus
    assume that  $t \in  T$.  If $Z  \subseteq \overline{X}$  then the
    alternative~\ref{itm:s-t} holds,  with $Y:=Z$  and $D:=F$.  It now
    remains the case in which $Z \subset X$ and $t \in T$.

    Let $u$ be the vertex of $T-t$.  One of $u$ and $t$ is in $Z$.
    As $t \notin X$, it follows  that $u \in Z$.  Let $Z':=(Z-u) +
    t$.   The  cuts  $F$  and  $F':=\partial(Z')$  are  members  of  a
    2-separation cut  pair of  $G$.  Let  $H$ be  the $F'$-contraction
    $\Contra{G}{Z'}{z'}$  of $G$,  let  $X_H:=(X -  {Z})+u$ and  let
    $C_H:=\partial(X_H)$.  See Figure~\ref{fig:ch-in-both}.

    \begin{figure}[!ht]
      \centering
      \psfrag{X}{$X$}
      \psfrag{Y}{$Z$}
      \psfrag{D}{$F$}
      \psfrag{Yp}{$Z'$}
      \psfrag{Dp}{$F'$}
      \psfrag{Ylv1}{$Z-u$}
      \psfrag{XlY}{$X-Z$}
      \psfrag{Xblt}{$\overline{X}-t$}
      \psfrag{XH}{$X_H$}
      \psfrag{CH}{$C_H$}
      \psfrag{v1}{$u$}
      \psfrag{C}{$C$}
      \psfrag{t}{$t$}
      \newcommand{\h}{150pt}
      \subfigure [$G$] {
 	\includegraphics [height = \h] {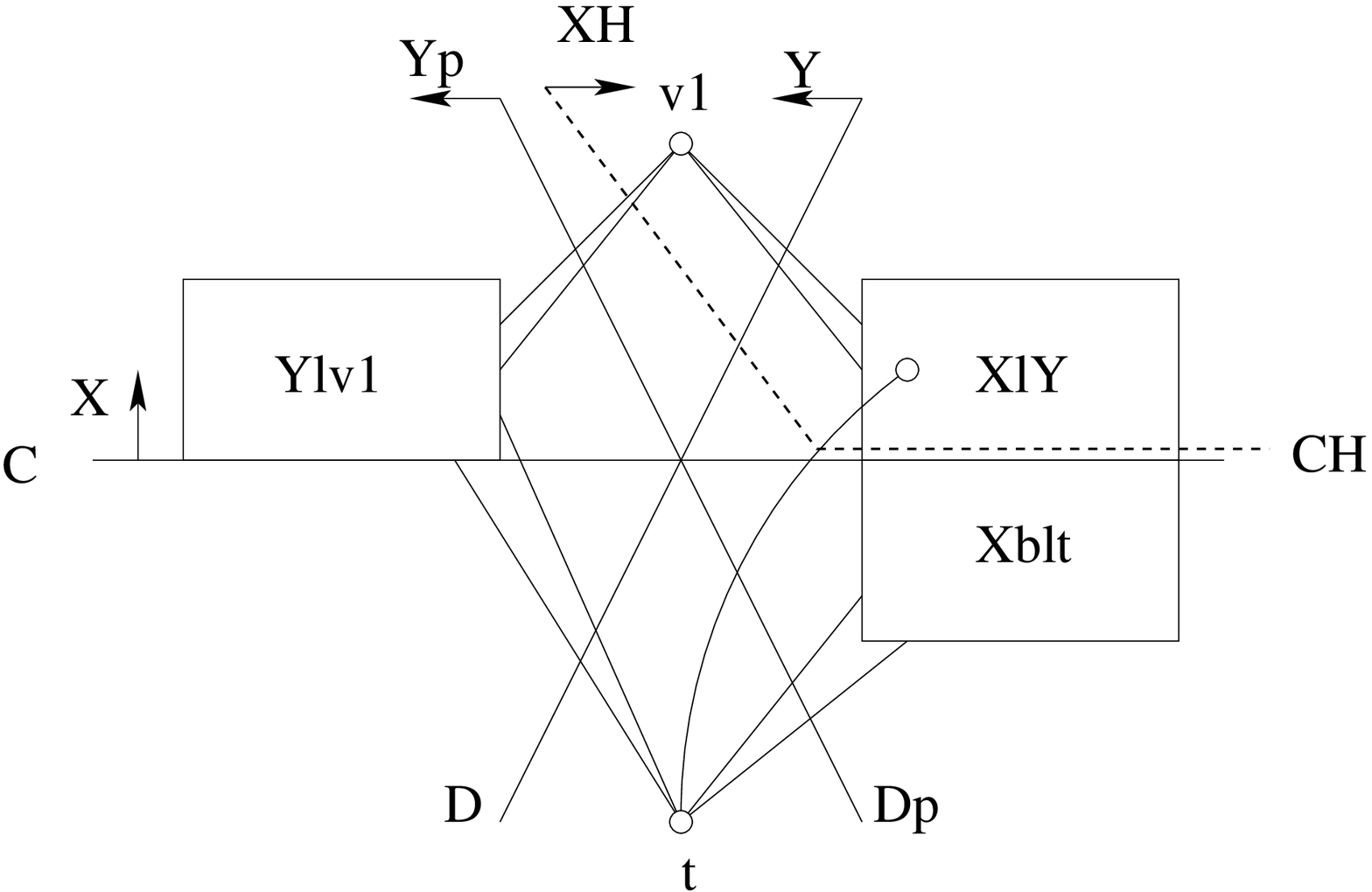}
 	\label{fig:ch-in-g}
      }
      \psfrag{C}{~}
      \psfrag{t}{$z'$}
      \subfigure [$H$] {
 	\includegraphics [height = \h] {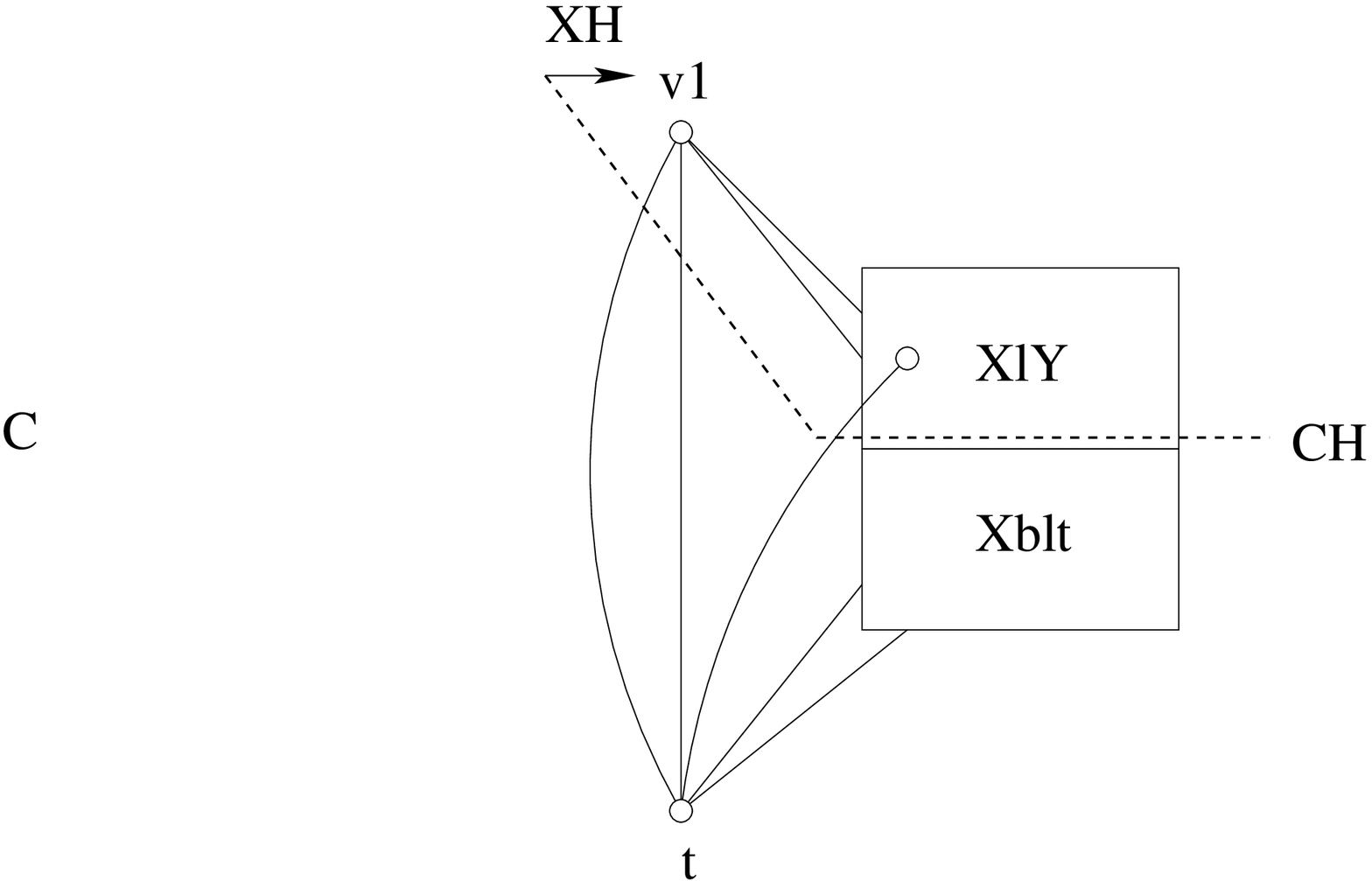}
 	\label{fig:ch-in-h}
      }
      \caption{The cut $C_H$ in $G$ and in $H$}
      \label{fig:ch-in-both}
    \end{figure}

    We plan  now to apply  the induction hypothesis, with  $H$, $X_H$,
    $C_H$ and $z'$ playing respectively the roles of $G$, $X$, $C$ and
    $t$.  The cuts $C$ and $F'$  cross, the intersection of the shores
    $X$ and  $\overline{Z'}$ is odd  and equal  to the shore  $X_H$ of
    $C_H$.      The     cut     $C_H$      is     tight     in     $G$
    (Proposition~\ref{pro:I-U}). As $C_H$ and  $F'$ are laminar, $C_H$
    is   tight   in   $H$   (Proposition~\ref{pro:tight-contraction}).
    Moreover, as $Z \subset X$, it  follows that $C_H$ is a nontrivial
    tight  cut  of $H$.   As  $|V(H)|  <  |V(G)|$,  we may  infer,  by
    hypothesis, that the assertion of the Main Theorem holds for every
    graph  having  $|V(H)|$  or  fewer vertices.   Moreover,  $z'  \in
    \overline{X_H}$.  We  now apply  the induction hypothesis  to $H$,
    $X_H$, $C_H$ and $z'$ playing  respectively the roles of $G$, $X$,
    $C$   and  $t$.   We  consider   separately  the   three  possible
    alternatives.

    \setcounter{cas}{0}

    \begin{cas}
      The  graph $H$  has a  2-separation which  does not  contain the
      vertex $z'$.
    \end{cas}
    Let $S$ be a 2-separation of  $H$ that does not contain the vertex
    $z'$.  By Lemma~\ref{lem:inheritance},  $S$ is  a 2-separation  of
    $G$.  Clearly,  $S$   does  not  contain  the   vertex  $t$.   The
    alternative~\ref{itm:s-no-t} of the assertion holds.

    \begin{cas}
      The graph $H$ has a  2-separation $S_H$ that contains the vertex
      $z'$, associated  with a 2-separation  cut $D_H:=\partial(Y_H)$,
      such that $Y_H \subseteq \overline{X_H}$.
    \end{cas}
    Let $W_H:=Y_H-S_H$. As $z' \in S_H$, the
    vertex $z'$ is not in  $W_H$. As $Y_H \subseteq \overline{X}$, the
    vertex $u$  is not  in $W_H$.   In sum,  $W_H$ and  $\{u,z'\}$ are
    disjoint. Let $S:=(S_H-z')+t$ and let $Y:=W_H+t = (Y_H - S_H) + t$.
    By Lemma~\ref{lem:inheritance}, every (even) component of $H[W_H]$
    is  a  component of  $G-S$  and  $S$  is  a 2-separation  of  $G$.
    Moreover, the vertex $t$ is in  $S$, the cut $D:=\partial(Y)$ is a
    2-separation cut of $G$ associated with $S$ and its shore $Y$ is a
    subset  of $\overline{X}$.   The alternative~\ref{itm:s-t}  of the
    assertion holds.

    \begin{cas}
      The graph $H$ has a $C_H$-sheltered nontrivial barrier.
    \end{cas}
    Let $S_H$ be a $C_H$-sheltered  nontrivial barrier of $H$.  We now
    apply Lemma~\ref{lem:inheritance}.   Let $S$ be as  defined in the
    statement of  that Lemma. Then,  $S$ is a (nontrivial)  barrier of
    $G$. If  $S_H$ is  a subset  of $X_H$  then $S=S_H$  and $S$  is a
    subset of  $X$. If $S_H$  is a subset of  $\overline{X_H}-z'$ then
    $S=S_H$ and $S$ is a  subset of $\overline{X}$.  Finally, if $S_H$
    is a subset  of $\overline{X}$ and $z' \in S_H$  then $S$ is equal
    to  $(S_H-z')+t$  and  is  a subset  of  $\overline{X}$.   In  all
    alternatives, $S$  is a  $C$-sheltered nontrivial barrier  of $G$,
    hence  alternative~\ref{itm:c-sheltered} of  the assertion  of the
    Lemma holds.

    \medskip
    The proof of the Lemma is complete.
  \end{proof}

  \section{Proof of the Main Theorem}
  \label{sec:pf}

  \begin{proof}
    Let $G$ be a matching covered  graph and let $C:=\partial(X)$ be a
    nontrivial  tight cut  of  $G:=(V,E)$.  Assume,  as the  induction
    hypothesis, that  the assertion  holds for every  matching covered
    graph having  fewer than $|V|$  vertices. We shall now  prove that
    either $G$ has a 2-separation cut  having a shore that is a subset
    of  a  shore  of  $C$   or  $G$  has  a  $C$-sheltered  nontrivial
    barrier.

    As $C$  is nontrivial  and tight,  we infer  from the  ELP Theorem
    (Theorem~\ref{thm:elp}) that  $G$ either  has a 2-separation  or a
    nontrivial barrier. We consider these possibilities separately.

    \setcounter{cas}{0}

    \begin{cas}
      \label{cas:no-2-sep}
      The graph $G$ does not have 2-separations.
    \end{cas}
    In this case,  the graph $G$ has a nontrival  barrier, $B$. If $B$
    is               $C$-avoiding               then,               by
    Lemma~\ref{lem:barrier-sheltered}\ref{itm:barrier-sheltered-nontrivial},
    one  of $B  \cap  X$ and  $B \cap  \overline{X}$  is a  nontrivial
    ($C$-sheltered) barrier of $G$.  We may thus assume that $G-B$ has
    a component, $G[Y]$,  such that $\partial(Y)$ and  $C$ cross.  Let
    $H:=\Contra{G}{\overline{Y}}{\overline{y}}$         and        let
    $D:=\partial(Y)$.   Adjust  notation,  by interchanging  $X$  with
    $\overline{X}$ if necessary, so that $|X \cap Y|$ is odd.

    Let $I:=\partial(X \cap Y)$,  let $U := \partial(\overline{X} \cap
    \overline{Y})$.   The cuts  $I$ and  $U$  are both  tight in  $G$.
    Moreover, no edge  of $G$ joins a vertex of  $\overline{X} \cap Y$
    to      a      vertex      in      $X      \cap      \overline{Y}$
    (Proposition~\ref{pro:I-U}). As  $|X \cap Y|$  is odd, $H$  has an
    odd number of  vertices in $X$. Moreover, $Y$ is  neither a subset
    of  $X$ nor  a superset  of  $\overline{X}$. The  subgraph of  $G$
    induced by  $\overline{X}$ is  connected.  Thus,  $G$ has  an edge
    joining a vertex  of $\overline{X} \cap Y$ to a  vertex in $B \cap
    \overline{X}$. In  sum, $H$ has an  odd number of vertices  in $X$
    and has  a vertex adjacent to  a vertex in $B  \cap \overline{X}$.
    By   Lemma~\ref{lem:barrier-sheltered}\ref{itm:barrier-sheltered},
    $B  \cap  \overline{X}$  is  a  (possibly  trivial)  $C$-sheltered
    barrier of $G$.  If $B \cap  \overline{X}$ is not a singleton then
    the assertion of  the Theorem holds.  We may thus  assume that $|B
    \cap \overline{X}|  = 1$.  Let $u$  be the only vertex  of $B \cap
    \overline{X}$.                                                  By
    Lemma~\ref{lem:barrier-sheltered}\ref{itm:barrier-accounting},
    every component of $G-B$ has an odd number of vertices in $X$.

    \begin{pro}
      The cut $I$ is nontrivial.
    \end{pro}
    \begin{proof}
      Assume,  to  the contrary,  that  $X  \cap  Y$ is  a  singleton,
      $\{v\}$.  No edge of $G$ joins a vertex of $X \cap \overline{Y}$
      to  a vertex  of $\overline{X}  \cap  Y$.  Thus,  every edge  of
      $\partial(\overline{X}  \cap Y)$  is incident  with a  vertex in
      $\{u,v\}$.  We conclude that $\{u,v\}$ is a 2-separation of $G$,
      in  contradiction  to  the  hypothesis  that  $G$  is  free  of
      2-separations. (In fact, the cut $\partial((\overline{X} \cap Y)
      + u)$ is a 2-separation cut of $G$ associated with $\{u,v\}$.)
    \end{proof}

    We  now apply  Lemma~\ref{lem:t}, with  $H$, $X  \cap Y$,  $I$ and
    $\overline{y}$ playing,  respectively, the roles of  $G$, $X$, $C$
    and $t$. We then deduce that one of the following possibilities
    holds:
    \begin{enumerate}\rmenum
    \item
      either the  graph $H$ has  a 2-separation that does  not contain
      the vertex $\overline{y}$, or
    \item
      $H$  has  a  2-separation,   $S_H$,  that  contains  the  vertex
      $\overline{y}$, associated with a cut $D_H:=\partial(Y_H)$, such
      that $Y_H \subseteq (\overline{X} \cap Y) + \overline{y}$, or
    \item
      the graph $H$ has a nontrivial $I$-sheltered barrier, $B_H$.
    \end{enumerate}

    We shall  now eliminate the  two first possibilities.   Assume, to
    the contrary,  that $H$ has  a 2-separation, $S_1$, that  does not
    contain the vertex $\overline{y}$. One of the (even) components of
    $H-S_1$, $K_1$, does not contain  the vertex $\overline{y}$ and is
    a  proper  subgraph of  $H$.   In  that  case,  $K_1$ is  an  even
    component of $G-S_1$.  By Corollary~\ref{cor:barrier}, $G-S_1$ has
    no odd components hence $S_1$ is  a 2-separation of $G$. This is a
    contradiction to the hypothesis that $G$ is free of 2-separations.

    Assume, to the contrary, that  $H$ has a 2-separation, $S_H$, that
    contains  the  vertex  $\overline{y}$,  associated  with  the  cut
    $D_H:=\partial(Y_H)$, where $Y_H \subseteq (\overline{X} \cap Y) +
    \overline{y}$. Let $v$ denote the vertex of $S_H-\overline{y}$ and
    let $K$ be a component  of $H[Y_H-S_H]$. Necessarily, $K$ is even.
    The  set  $V(K)$ is  a  subset  of  $\overline{X} \cap  Y$,  hence
    $\partial_H(V(K))=\partial_G(V(K))$.   Let  $e$   be  an  edge  of
    $\partial(V(K))$  that is  not incident  with $v$.   Then, $e$  is
    incident with  $\overline{y}$ in $H$,  hence, in $G$, $e$  joins a
    vertex   of  $\overline{X}   \cap  Y$   to  a   vertex,  $w$,   of
    $\overline{Y}$.  Necessarily $w=u$, hence $K$ is an even component
    of  $G-\{u,v\}$.  By  Corollary~\ref{cor:barrier}, $\{u,v\}$  is a
    2-separation of $G$, again a  contradiction to the hypothesis that
    $G$ is free of 2-separations.

    We deduce that $H$ has  a nontrivial $I$-sheltered barrier, $B_H$.
    If $B_H$  is a subset  of $X \cap Y$  or of $\overline{X}  \cap Y$
    then  $B_H$ is  $C$-sheltered, and  the assertion  of the  Theorem
    holds. We may thus assume that $B_H$ is a subset of $(\overline{X}
    \cap    Y)   +    \overline{y}$   that    contains   the    vertex
    $\overline{y}$. In that case, the  set $B_G:= (B_H - \overline{y})
    \cup B$  is a barrier  of $G$. One  of the components  of $H-B_H$,
    say,  $K$,   contains  all  the   vertices  of  $X  \cap   Y$,  by
    Proposition~\ref{pro:sheltered}. As $\overline{y} \in B_H$, $K$ is
    also a component  of $G-B_G$ and $V(K)  \cap X = X  \cap Y$, hence
    $K$ contains an odd number of vertices in $X$. Moreover, as $H$ is
    2-connected, $K$ has  vertices adjacent to at least  one vertex of
    $B_H   -  \overline{y}$,   which  is   a  vertex   of  $B_G   \cap
    \overline{X}$. Finally, $B_G \cap  \overline{X}$ is nontrivial. By
    Lemma~\ref{lem:barrier-sheltered}\ref{itm:barrier-sheltered}, $B_G
    \cap \overline{X}$ is a nontrivial $C$-sheltered barrier of $G$.

    \medskip
    The analysis of Case~\ref{cas:no-2-sep}  is complete.  We may thus
    assume that

    \begin{displaymath}
      \mbox{\underline{$G$ has 2-separations}.}
    \end{displaymath}

    If $G$ has a $C$-sheltered  2-separation then the assertion holds,
    by  Proposition~\ref{pro:sheltered}.  We  may  thus assume  that
    \begin{displaymath}
      \mbox{\underline{each 2-separation  of $G$ contains a  vertex in
	  each shore of $C$}.}
    \end{displaymath}

    \smallskip\noindent Let $S$  be a 2-separation of $G$  and let $K$
    be a component of $G-S$.  As $K$ is even, it follows that $|X \cap
    V(K)| \equiv |\overline{X}  \cap V(K)| \pmod{2}$. We  say that $K$
    is   \emph{balanced}   if   $|X    \cap   V(K)|$   is   even   and
    \emph{unbalanced}, otherwise.

    \begin{pro}
      \label{pro:odd=2}
      Let  $S$  be a  2-separation  of  $G$  such  that $G-S$  has  an
      unbalanced component.  Then, $G-S$ has precisely two components,
      both of which  are unbalanced.  Moreover, every edge  of $C$ has
      both     ends    in     a     component     of    $G-S$     (See
      Figure~\ref{fig:unbalanced}).
      \begin{figure}[!ht]
	\psfrag{C}{$C$}
	\psfrag{X}{$X$}
	\psfrag{Y}{$Y$}
	\psfrag{Xl}{$X \cap Y$}
	\psfrag{Xb}{$(\overline{X} \cap Y)-u_1$}
	\psfrag{L1}{$L_1$}
	\psfrag{L2}{$L_2$}
	\psfrag{D}{$D$}
	\psfrag{I}{$I$}
	\psfrag{u1}{$u_1$}
	\psfrag{u2}{$u_2$}
	\psfrag{yb}{$\overline{y}$}
	\psfrag{G}{{\small(a) $G$}}
	\psfrag{H}{{\small(b) $H$}}
	\includegraphics [width=\textwidth] {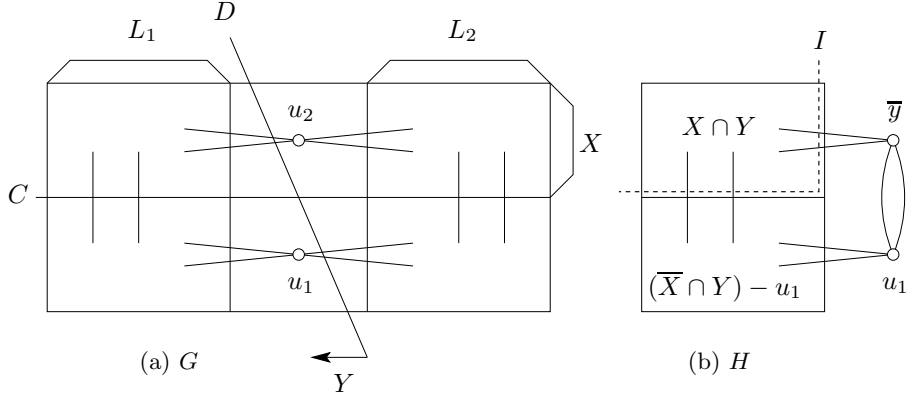}
	\caption{The    components    $G[L_1]$   and    $G[L_2]$    of
	  $G-\{u_1,u_2\}$ are unbalanced}
	\label{fig:unbalanced}
      \end{figure}
    \end{pro}
    \begin{proof}
      Suppose that  $G-S$ has an unbalanced  component, $K_1:=G[L_1]$.
      The shore $X$ of $C$ is odd and it contains precisely one vertex
      of $S$.  Moreover, $|L_1 \cap X|$  is also odd.  Thus, $|X - L_1
      - S|$ is odd. It follows that $G-S$ has an unbalanced component,
      $K_2:=G[L_2]$, distinct from $K_1$.

      Assume,  to  the  contrary,  that $G-S$  contains  a  component,
      $K_3:=G[L_3]$,  not necessarily  unbalanced,  but distinct  from
      both $K_1$  and $K_2$. Let  $e$ be  an edge of  $\partial(L_3)$
      and  let $M$ be  a perfect  matching      of $G$  that contains
      edge $e$. Necessarily, $e$ is incident  with a vertex of $S$. As
      $L_3$  is even,  $M$ contains  also an  edge in  $\partial(L_3)$
      which is incident with the other vertex of $S$.  For $i=1,2$, as
      $|X \cap L_i|$  is odd, the cut $\partial(X  \cap L_i)$ contains
      an edge in  $M$ and that edge  is not incident with  a vertex of
      $S$, hence it is an edge having both ends in $L_i$, therefore it
      is an edge  of $C$.  We conclude that $M$  contains more than one
      edge  in $C$,  a contradiction  to  the hypothesis  that $C$  is
      tight.  Indeed, $K_1$  and $K_2$ are the only  two components of
      $G-S$ and they are both unbalanced.

      Let $u_1$ and $u_2$ be two  vertices of $S$.  Let $Y:=L_1 + u_1$
      be the shore of a cut associated with $S$ such that $|X \cap Y|$
      is odd.  We  have assumed that $S$ has one  vertex in each shore
      of  $C$, hence  $u_1 \in  \overline{X}$  and $u_2  \in X$.   Let
      $D:=\partial(Y)$ (See Figure~\ref{fig:unbalanced}). The cuts $C$
      and $D$ are tight  and cross.  No edge of $G$  joins a vertex in
      $(\overline{X} \cap L_1)  + u_1$ to a vertex in  $(X \cap L_2) +
      u_2$ (Proposition~\ref{pro:I-U}).   By symmetry, no edge  of $G$
      joins  a  vertex  in  $(X  \cap  L_1) +  u_2$  to  a  vertex  in
      $(\overline{X} \cap L_2) + u_1$.
    \end{proof}

    Let $S$  be a 2-separation  of $G$.  A  component $K$ of  $G-S$ is
    \emph{good} if either $K$ is balanced  or    each vertex of $S$ is
    adjacent to two or more vertices  of $K$. Let $\mathcal{S}$ be the
    collection of 2-separations of $G$.  For each $S \in \mathcal{S}$,
    let $\mathcal{F}(S)$ be  the set of good components  of $G-S$. Let
    $$\mathcal{F} := \bigcup_{S \in \mathcal{S}} \mathcal{F}(S).$$

    \begin{cas}
      \label{cas:no-good}
      The collection $\mathcal{F}$ is empty.
    \end{cas}
    Let  $S$ be  a 2-separation  of $G$.  The hypothesis  of the  case
    implies that the components     of  $G-S$   are    unbalanced. By
    Proposition~\ref{pro:odd=2}, $G-S$  consists of precisely two
    components,  $K_i:=G[L_i]$, $i=1,2$.  We  have  assumed that  each
    shore of $C$ contains a vertex of $S$.  Let $u_1$ be the vertex of
    $S$ in  $\overline{X}$ and let $u_2$  be the vertex of  $S$ in $X$
    (Figure~\ref{fig:unbalanced}).

    The hypothesis of the case also implies that for $i = 1,2$, one of
    $u_1$ and $u_2$  is adjacent only to one vertex  of $L_i$.  Adjust
    notation so  that $u_1$ is adjacent  to only one vertex  of $L_1$,
    says $v_1$.

    \begin{pro}
      The vertex $u_1$ is adjacent only to one vertex of $L_2$.
    \end{pro}
    \begin{proof}
      Suppose, to  the contrary, that  $u_1$ is adjacent to  more than
      one vertex  in $L_2$. By  the hypothesis  of the case,  $u_2$ is
      adjacent  only  to  one  vertex   in  $L_2$,  say,  $v_2$.   Let
      $T:=\{v_1,v_2\}$.   No   edge  of   $G-T$  joins  a   vertex  of
      $Z_1:=(L_1-v_1)+u_2$ to a vertex of $Z_2:=(L_2-v_2)+u_1$.  Thus,
      $G-T$ has two or more components.

      Let us  now prove that each  component of $G-T$ is  even. If $G-T$
      has an  odd component  then, by  parity, it has  at least  two odd
      components.  In that  case, by  Corollary~\ref{cor:barrier}, $G-T$
      has precisely two components, both odd.  But if $G-T$ has only two
      components  then they  are $G[Z_1]$  and $G[Z_2]$,  both even.  We
      conclude that each component of $G-T$ is even.

      Let us now  prove that each component of $G-T$  has an even number
      of vertices in each shore of $C$. If $G-T$ has only two components
      then they are $G[Z_1]$ and $G[Z_2]$,  and both have an even number
      of vertices  in each  shore of $C$.  Alternatively, if  $G-T$ has
      more than two  components then again each of  these components has
      an   even   number   of   vertices   in   each   shore   of   $C$,
      by Proposition~\ref{pro:odd=2}.

      As  each  component  of  $G-T$  is  even,  the  pair  $T$  is  a
      2-separation of $G$. Each component  of $G-T$ has an even number
      of vertices in each shore of $C$. We conclude that $\mathcal{F}$
      is nonempty, in contradiction to the hypothesis of the case.
    \end{proof}

    In sum,  for $i=1,2$,  the vertex  $u_1$ is  adjacent to  only one
    vertex  of   $L_i$,  say,  $v_i$.  Clearly,   $\{v_1,v_2\}$  is  a
    $C$-sheltered nontrivial barrier of $G$.

    \begin{cas}
      \label{cas:good}
      The collection $\mathcal{F}$ is nonempty.
    \end{cas}
    Let   $K_1$  be   a  minimal   component  in   $\mathcal{F}$,  let
    $L_1:=V(K_1)$ and let  $S$ be the associated  2-separation of $G$.
    Let $u_1$ and $u_2$ be two vertices of $S$.  Let $Y:=L_1 + u_1$ be
    the shore of a  cut associated with $S$ such that  $|X \cap Y|$ is
    odd and  let $D:=\partial(Y)$.  We  have assumed that $S$  has one
    vertex in each shore of $C$.  If $K_1$ is unbalanced then $u_1 \in
    \overline{X}$       and        $u_2       \in        X$       (See
    Figure~\ref{fig:unbalanced}). Alternatively, if  $K_1$ is balanced
    then   $u_1    \in   X$   and   $u_2    \in   \overline{X}$   (See
    Figure~\ref{fig:balanced}).                                    Let
    $H:=\Contra{G}{\overline{Y}}{\overline{y}}$ and let $I:=\partial(X
    \cap    Y)$.     The    cuts    $C$    and     $D$    cross.    By
    Proposition~\ref{pro:I-U}, the cut $I$ is  tight and no edge joins
    a  vertex  of  $\overline{X}  \cap  Y$ to  a  vertex  of  $X  \cap
    \overline{Y}$.

    \begin{figure}[!ht]
      \psfrag{C}{$C$}
      \psfrag{X}{$X$}
      \psfrag{Y}{$Y$}
      \psfrag{Xl}{$(X \cap Y)-u_1$}
      \psfrag{Xb}{$\overline{X} \cap Y$}
      \psfrag{L1}{$L_1$}
      \psfrag{D}{$D$}
      \psfrag{I}{$I$}
      \psfrag{u1}{$u_1$}
      \psfrag{u2}{$u_2$}
      \psfrag{yb}{$\overline{y}$}
      \psfrag{G}{{\small(a) $G$}}
      \psfrag{H}{{\small(b) $H$}}
      \includegraphics [width=\textwidth] {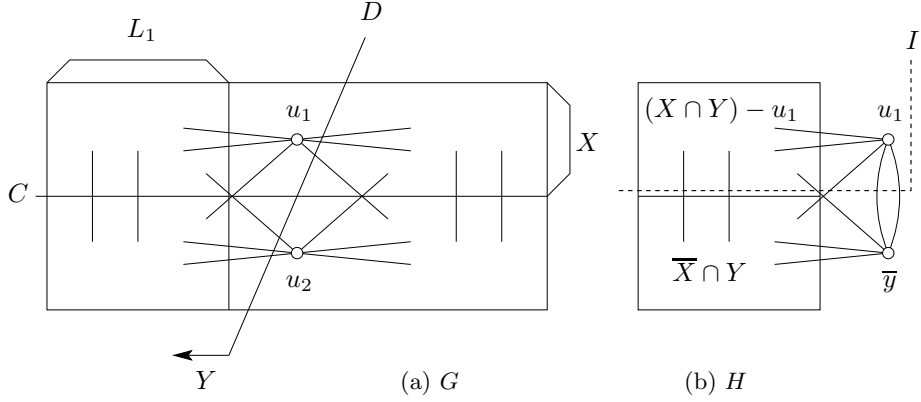}
      \caption{The components $G[L_1]$ of $G-\{u_1,u_2\}$ is balanced}
      \label{fig:balanced}
    \end{figure}

    \begin{sub}
      The cut $I$ is trivial.
    \end{sub}
    Suppose  that $I$  is trivial.  If $K_1$  is unbalanced,  then, as
    $K_1$ is good, it has two or more vertices adjacent to vertices of
    $S$, hence  $I$ is nontrivial.  We deduce that $K_1$  is balanced.
    In  that   case,  $X  \cap  Y   =  \{u_1\}$  and  every   edge  of
    $\partial(\overline{X}  \cap  Y)$ is  incident  with  a vertex  in
    $S$. It follows that the shore  $Y':= (\overline{X} \cap Y) + u_2$
    of   the   associated   cut   $\partial(Y')$  is   a   subset   of
    $\overline{X}$. The assertion of the Theorem holds in this case.
    We may thus assume that
    \begin{displaymath}
      \mbox{\underline{The cut $I$ is nontrivial}.}
    \end{displaymath}
    The  cuts $D$  and $I$  are  laminar, hence  $I$ is  tight in  $H$
    (Proposition~\ref{pro:tight-contraction}).    We  now   apply  the
    induction hypothesis,  with $H$  and $I$ playing  respectively the
    roles of $G$ and $C$.

    \begin{sub}
      \label{cas:barrier}
      The graph $H$ has a nontrivial $I$-sheltered  barrier.
    \end{sub}
    Let $S_H$  be a nontrivial  $I$-sheltered barrier of $H$.   Let us
    now   apply   Lemma~\ref{lem:inheritance}    to   $H$,   $Y$   and
    $\overline{y}$  playing respectively  the  roles of  $G$, $X$  and
    $\overline{x}$. Let  $S$ be  the set defined  in the  statement of
    that Lemma. By the item~\ref{itm:inheritance} of the Lemma, $S$ is
    a (nontrivial) barrier of $G$.   If $\overline{y} \notin S_H$ then
    $S$  is equal  to  $S_H$ and  is  a subset  of $X  \cap  Y$ or  of
    $\overline{X} \cap  Y$, hence $S$  is $C$-sheltered.  We  may thus
    assume   that    $\overline{y}   \in    S_H$,   in    which   case
    $S=(S_H-\overline{y})   +  u_2$   and   $S_H$  is   a  subset   of
    $(\overline{X} \cap Y) + \overline{y}$.  If $K_1$ is balanced then
    $u_2 \in \overline{X}$ (Figure~\ref{fig:balanced}), hence $S$ is a
    subset  of  $\overline{X}$.  We  may  thus  assume that  $K_1$  is
    unbalanced.

    In  $H$, some  component of  $H-S_H$, say,  $W$, contains  all the
    vertices of $X \cap Y$, by Proposition~\ref{pro:sheltered}. By the
    2-connectivity of $H$, $W$ contains  a vertex adjacent to a vertex
    of    $S_H   \cap    \overline{X}   \cap    Y$.    Moreover,    by
    Lemma~\ref{lem:inheritance},  $W$ is  a  component  of $G-S$.   By
    Lemma~\ref{lem:barrier-sheltered}\ref{itm:barrier-sheltered},  the
    set  $S -  u_2$, which  is equal  to $S  \cap \overline{X}$,  is a
    barrier of  $G$. If $S  \cap \overline{X}$ is nontrivial  then the
    assertion of the Theorem holds.

    We may thus assume that $S  \cap \overline{X}$ is trivial, let $v$
    be its  only vertex. Let  $L$ be  a component of  $S-S_H$ distinct
    from $W$. Then, $V(L)\subset\overline{X} \cap  Y$. The only vertex
    of $\overline{X}\cap Y$ adjacent to $\overline{y}$ is the vertex $u_1$,
    hence $L$ contains $u_1$. We deduce  that $L$ and $W$ are the only
    two components of $H-S_H$. As  $K_1$ is unbalanced and good, $u_1$
    is adjacent  to two or  more vertices of $\overline{X}  \cap L_1$,
    hence $L$ is  nontrivial.  The graph $L-u_1$ is  a proper nonempty
    subgraph   of  $G$,   hence   the  graph   $G-\{u_1,v\}$  is   not
    connected. By Corollary~\ref{cor:barrier}, $\{u_1,v\}$ is either a
    barrier or a 2-separation of $G$. As $\{u_1,v\}$ is $C$-sheltered,
    the      assertion     of      the      Theorem     holds,      by
    Proposition~\ref{pro:sheltered}.       The       analysis       of
    Case~\ref{cas:barrier} is complete.

    \begin{sub}
      \label{cas:2-sep}
      The  graph $H$  has  a  2-separation cut,  $D_H:=\partial(Z_H)$,
      associated  with a  2-separation  $S_H$, such  that  $Z_H$ is  a
      subset of a shore of $I$.
    \end{sub}
    Let  us  now apply  Lemma~\ref{lem:inheritance}  to  $H$, $Y$  and
    $\overline{y}$  playing respectively  the  roles of  $G$, $X$  and
    $\overline{x}$. Let  $S$ be  the set defined  in the  statement of
    that Lemma. By the item~\ref{itm:inheritance} of the Lemma, $S$ is
    a 2-separation of  $G$.  We have assumed that one  of the vertices
    of $S$ is in $X$, the other is in $\overline{X}$. Let $w_1$ be the
    vertex  of $S$  in $X$,  and let  $w_2$ be  the vertex  of $S$  in
    $\overline{X}$.

    Let $W_H:=Z_H-S_H$.  Then, $W_H$ is a  proper subset of a shore of
    $I$.  If $W_H$ is  a subset of a shore of $C$,  then, as $S$ meets
    both shores of  $C$, it follows that one of  the 2-separation cuts
    of $G$ associated with $S$ has a shore that is a subset of $C$. We
    may thus assume that $W_H$ meets  both shores of $C$.  As $W_H$ is
    a subset of a shore of $I$, it follows that $\overline{y} \in W_H$
    and  $W_H$ is  a  subset of  the shore  $(\overline{X}  \cap Y)  +
    \overline{y}$. Then, $S=S_H$.

    If $K_1$  is unbalanced then,  by definition of $K_1$,  the vertex
    $u_2$ is adjacent  to at least two vertices of  $X \cap Y$.  Thus,
    both  vertices  of  $S_H$  are  in   $X  \cap  Y$,  hence  $S$  is
    $C$-sheltered, a case already considered.  We may thus assume that
    $K_1$ is balanced.  The subgraph  of $G$ induced by $\overline{Y}$
    is connected  (Corollary~\ref{cor:shores-connected}).  One  of the
    components  of $G-S$  is  the  graph $G[W]$,  where  $W  = (W_H  -
    \overline{y}) \cup \overline{Y}$.   As $W_H$ is even,  then so too
    is $W$.   The vertices  $u_1$ and  $\overline{y}$ are  adjacent in
    $H$,  therefore $u_1  \in S_H  \cup W_H$,  hence $u_1  \in S  \cup
    W$. We conclude that every component of $G-S$ distinct from $G[W]$
    is   a   proper   subgraph    of   $K_1$.    Moreover,   the   set
    $T_1:=W_H-\overline{y}$ is an odd subset of $\overline{X} \cap Y$,
    and $T_2:=\overline{X} \cap \overline{Y}$ is also odd. Clearly, $W
    \cap \overline{X} = T_1 \cup  T_2$, hence $W \cap \overline{X}$ is
    even. That is, $W$ is balanced.  By Proposition~\ref{pro:odd=2},
    every component of $G-S$ is balanced. In particular, the components
    of $G-S$  distinct  from  $W$  are  proper subgraphs of  $K_1$,  in
    contradiction  to  the  minimality   of  $K_1$.

    \bigskip
    The  analysis of  Case~\ref{cas:2-sep} concludes  the analysis  of
    Case~\ref{cas:good} and completes the proof of the Main Theorem.
  \end{proof}

  \bibliography{references}

\end{document}